\documentclass[11pt]{amsart}

\usepackage{amsmath,amsthm}
\usepackage{amsfonts}
\usepackage{amssymb}

\usepackage[nobysame]{amsrefs}
\usepackage[citecolor=red,colorlinks=true]{hyperref}
\usepackage{fouriernc}

\newcommand{\set}[1]{\,\left\{#1\right\}}

\newtheorem{theorem}{Theorem}

\newtheorem{lemma}[theorem]{Lemma}
\newtheorem{proposition}[theorem]{Proposition}

\newtheorem{question}{Question}

\newcommand{\ZZ}{\mathbb{Z}}
\newcommand{\CC}{\mathbb{C}}

\title{Uniformly bounded representations and exact groups}
\date{22 April 2013}

\author{Kate Juschenko}
\address{Vanderbilt University, Department of Mathematics, 1326 Stevenson Center, Nashville, TN 37240, USA}
\email{kate.juschenko@gmail.com}

\author{Piotr W. Nowak}

\thanks{The second author was partially supported by the Foundation for Polish Science\\\\
Coressponding author's email address: pnowak@mimuw.edu.pl}
\address{Instytut Matematyczny Polskiej Akademii Nauk, \'{S}niadeckich 8, 00-956 Warszawa, Poland}
\address{Instytut Matematyki, Uniwersytet Warszawski, Banacha 2, 02-097 Warszawa, Poland}
\email{pnowak@mimuw.edu.pl}

\begin{document}

\begin{abstract}
We characterize groups with Guoliang Yu's property A (i.e., exact groups) by the existence of a family
of uniformly 
bounded representations which approximate  the trivial representation.
\end{abstract}

\maketitle

Property A  is a large scale geometric property that can be viewed as a weak counterpart of amenability.
It was shown in  \cite{yu}, that for a finitely generated group property A implies the Novikov conjecture.
It was also quickly realized that this notion has many other applications and interesting connections,
see \cite{nowak-yu-book,nowak-yu-notices}.

A well-known characterization of amenability states that the constant function 1 on $G$, as a coefficient of the trivial
representation, can be approximated by diagonal, finitely supported coefficients of the 
left regular representation of $G$ on $\ell_2(G)$. In this note we prove a counterpart of this result 
for groups with property A in terms of uniformly bounded representations. A representation $\pi$ of a group $G$ on a Hilbert space $H$
is said to be uniformly bounded if $\sup_{g\in G}\Vert\pi_g\Vert_{B(H)}<\infty$.

\begin{theorem}\label{theorem : main}
Let $G$ be a finitely generated group equipped with a word length function. $G$ has property A (i.e., $G$ is exact)  if and only if 
for every $\varepsilon>0$ there exists a uniformly bounded
representation $\pi$ of $G$ on a Hilbert space $H$, a vector $v\in H$ and a constant $S>0$ such that 
\begin{enumerate}
\item $\Vert \pi_g v\Vert=1$ for all $g\in G$,
\item $\vert 1-\langle \pi_gv,\pi_hv\rangle \vert\le\varepsilon$ if $\vert g^{-1}h\vert\le 1$,
\item $\langle \pi_gv,\pi_hv\rangle=0$ if $ \vert g^{-1}h\vert \ge S$.
\end{enumerate}
\end{theorem}
Alternatively, the second condition can be replaced by an almost-invariance condition: 
$\Vert \pi_gv-\pi_hv\Vert\le \varepsilon$ if $\vert g^{-1}h\vert\le 1$.
Another characterization of property A in this spirit, involving convergence for isometric representations 
on Hilbert $C^*$-modules was
studied in \cite{douglas-nowak}.

Recall that the Fell topology on the unitary dual is defined using convergence of coefficients of unitary representations.
Theorem \ref{theorem : main} states that the trivial representation can be approximated by
uniformly bounded representations, in a fashion similar to Fell's topology. 

Similar phenomena were considered by M.~Cowling \cite{cowling1,cowling2} in the case of the Lie group 
$\mathrm{Sp}(n,1)$.
Recall that $\mathrm{Sp}(n,1)$ has property (T), and thus the trivial representation is an isolated point  
among the equivalence classes of unitary representations
in the Fell topology. Cowling showed that nevertheless, for $\mathrm{Sp}(n,1)$ 
the trivial representation can be approximated
by uniformly bounded representations in a certain sense. 
Theorem \ref{theorem : main} gives a similar statement for all discrete groups with property A.
Recall that almost all known groups with property (T) are known to have property A. In particular, the groups $\operatorname{SL}_n(\ZZ)$, $n\ge 3$,
satisfy property A \cite{guentner-higson-weinberger}. 
 
Moreover, under a stronger assumption that the group has Hilbert space compression strictly greater than $1/2$ in the
sense of  \cite{guentner-kaminker}, we obtain a path of uniformly bounded
representations, whose coefficients continuously interpolate between the trivial and the left regular representation.

 Theorem 1 suggests
the possibility of negating property A using strengthened forms of Kazhdan's property that applies to uniformly bounded representations.
\begin{question}
Are there finitely generated groups satisfying a sufficiently strong version of property (T) for uniformly bounded representations, so that these groups cannot have property A?
\end{question}
Certain versions of such a property (T) for uniformly bounded representations were considered by Cowling \cite{cowling1,cowling2}, but they would 
not apply directly in our case. Construction of new examples of finitely generated groups without property A 
is a major open problem in coarse geometry, with possible applications in operator algebras, index theory and topology of manifolds.

\subsubsection*{Acknowledgements} We would like to thank Pierre de la Harpe and Guoliang Yu for helpful comments
as well as Pierre-Nicolas Jolissaint, Mikael de la Salle and Alain Valette for carefully reading the
first version of the paper and suggesting several significant improvements.

\section{Uniformly bounded representations and property A}

Let $H_0$ be a Hilbert space with scalar product $\langle \cdot,\cdot\rangle_0$, and let $T$ be a bounded, positive, self-adjoint operator on $H_0$.
We additionally assume that $T$ has a spectral gap; that is, there exists $\lambda>0$ such that 
\begin{equation}\label{equation : spectral gap}
\langle v,Tv\rangle_0 \ge \lambda \langle v,v\rangle_0
\end{equation}
for every $v\in H_0$. 

The operator $T$ induces a new inner product on $V$, the vector space underlying $H_0$, by the formula
$$\langle v,w\rangle_T=\langle v,Tw\rangle_0.$$
The norm $\Vert v\Vert_T$ induced by $\langle\cdot,\cdot\rangle_T$ is equivalent to the original norm on $H_0$, since
$$\lambda\Vert v\Vert^2_0\le \Vert v\Vert_T^2\le \Vert T\Vert_{B(H_0)}\  \Vert v\Vert_0^2.$$
Thus we obtain a new Hilbert space $H_T$ by equipping $V$ with the norm induced by $T$.
A unitary representation $\pi$ on $H_0$ naturally becomes a uniformly bounded representation on $H_T$.
More precisely, the norm of the representation satisfies 
$$\Vert \pi\Vert=\sup_{g\in G}\Vert\pi_g\Vert_{B(H_T)} \le \dfrac{\Vert T\Vert_{B(H_0)}}{\lambda}.$$
In the Hilbert space $H_T$, the representation $\pi$ satisfies 
$$\pi_g^*=T^{-1}\pi_{g^{-1}}T,$$
for every $g\in G$.

We will now relate property A to the existence of uniformly bounded representations with the desired properties. 
From the discussion on renormings via positive operators  we derive the following fact.
\begin{lemma}\label{lemma : positive operator}
Let $G$ be a finitely generated group equipped with a word length function and $\varepsilon>0$. If there exists a Hilbert space $H_0$,
a positive self-adjoint bounded operator $T$ of $H$ satisfying (\ref{equation : spectral gap}) , a unitary representation $\pi$, a unit vector $v\in H_0$ and $S>0$ such that
\begin{enumerate}
\item $\langle \pi_gv,T\pi_gv\rangle_H=1$ for every $g\in G$,
\item $\vert 1-\langle \pi_gv,T\pi_hv\rangle_H \vert\le \theta$ whenever $\vert g^{-1}h\vert\le 1$,
\item $\langle \pi_gv,T\pi_hv\rangle_H=0$ whenever $\vert g^{-1}h\vert \ge S$,
\end{enumerate}
then there exists a uniformly bounded representation $\pi$ of $G$ on a Hilbert space $H_T$ and $v\in H_T$, satisfying the conditions listed in Theorem \ref{theorem : main}.
\end{lemma}
\begin{proof}
Let $V$ denote the vector space underlying $H$. We equip $V$ with a scalar product $\langle v,w\rangle_T=\langle v,Tw\rangle_0$ and obtain the space $H_T$
as explained in the previous section. Viewing $\pi$ and $v$ with respect to this new norm gives the required properties. 
\end{proof}

Recall that a Hermitian kernel on a set $X$ is a function $K:X\times X\to \CC$ such that
$K(x,y)=\overline{K(y,x)}$. $K$ is said to be positive definite if for every finitely supported function $f:X\to \CC$ 
we have
$$\sum_{x,y\in X}K(x,y)f(x)f(y)\ge 0.$$
Positive definite kernels can be used to characterize 
property A, we use this description as the definition. We refer to 
\cite{nowak-yu-notices,nowak-yu-book,tu} for more details and other characterizations of property A.

\begin{theorem}[see \cite{tu}]\label{theorem : property A via kernels}\label{theorem : property A via p.d. kernels}
A discrete metric space $X$ has property A if and only if for every $R,\varepsilon>0$ there exists a Hermitian positive definite kernel $K:X\times X\to [0,1]$ and $S>0$, satisfying
\begin{enumerate}
\item $K(x,x)=1$ for every $x\in X$,
\item $\vert 1-K(x,y)\vert\le \varepsilon$ if $d(x,y)\le R$,
\item $K(x,y)=0$ if $d(x,y)\ge S$.
\end{enumerate}
\end{theorem}
For a finitely generated group $G$ we take $X$ to be $G$ with the word length metric. In that case it suffices to consider only $R=1$.
A Hermitian kernel $K$ on $X$ induces a self-adjoint 
linear operator on $\ell_2(X)$, denoted also by $K$, by viewing $K$ as a matrix over $X$. We will identify the operator 
with the kernel representing it.

\begin{lemma}\label{lemma : main lemma}
Let $G$ be a finitely generated group with Yu's property A. Then for every $\varepsilon>0$ there exists 
an operator $T$ of a Hilbert space $H$, a unitary representation $\pi$ of $G$ on $H$ and  
a unit vector $v\in H$, satisfying the conditions of lemma \ref{lemma : positive operator}.
\end{lemma}

\begin{proof}
Let $\varepsilon>0$. Given $K$ as in Theorem \ref{theorem : property A via kernels}, define an operator
$$T=\dfrac{1}{1+\varepsilon}(K+\varepsilon I),$$
where $I$ is the identity on $H$.

It is clear that since $K$ is a positive operator, $T$ is also positive. 
It is easy to check that since $T$ is represented by a kernel, which takes values in the interval $[0,1]$ and vanishes outside of a neighborhood of the diagonal, 
 $T$ is a bounded operator on $\ell_2(G)$. Finally, 
$$\langle v,Tv\rangle = \langle v,Kv\rangle +\varepsilon \langle v,v\rangle \ge \varepsilon\langle v,v\rangle.$$
Thus $T$ is a positive, self-adjoint, bounded operator of  $H_0=\ell_2(G)$ and it satisfies (\ref{equation : spectral gap}). Consequently
we can construct a new Hilbert space $H_T$, isomorphic to $\ell_2(G)$, as 
explained earlier.

Consider now $\pi$, the left regular representation of $G$ on $\ell_2(G)$, viewed as a representation on $H_T$. 
By the previous discussion, $\pi$ is a uniformly bounded representation on $H_T$.

Denote by  $\delta_g$  the Dirac mass at $g\in G$ and let $v=\delta_e$.
Whenever $g\neq h$  we have 
\begin{equation}\label{equation : relation between kernels T and K}
\langle\pi_gv,T\pi_hv\rangle=\dfrac{1}{1+\varepsilon} \langle \pi_gv,K\pi_hv\rangle=\dfrac{1}{1+\varepsilon} \langle \delta_g,K\delta_h\rangle=\dfrac{1}{1+\varepsilon}K(g,h),
\end{equation}
and 
$$\Vert \pi_gv\Vert_T=\langle \delta_g, T\delta_g\rangle=1.$$
For $g,h\in G$ such that $\vert g^{-1}h\vert=1$ we can estimate
\begin{eqnarray*}
\vert 1- \langle \pi_gv,T\pi_hv\rangle\vert&=&\vert 1- \langle \delta_g,T\delta_h\rangle\vert\\
&=&\vert 1-T(g,h)\vert\\ 
&=&\vert 1-\dfrac{1}{1+\varepsilon} K(g,h)\vert\\
&\le &\varepsilon+\dfrac{\varepsilon}{1+\varepsilon},
\end{eqnarray*}
by (\ref{equation : relation between kernels T and K}). Also, 
$$\langle \pi_gv,T\pi_hv\rangle=\dfrac{1}{1+\varepsilon}\langle \pi_gv,K\pi_hv\rangle=0,$$ 
whenever
$\vert g^{-1}h\vert\ge S$. 
Thus $T$, $\pi$ and $v$ satisfy the required conditions with $S$ and 
$\varepsilon'=\varepsilon+\dfrac{\varepsilon}{1+\varepsilon}\le 2\varepsilon$.
\end{proof}

We are now in the position to prove the main theorem.
\begin{proof}[Proof of Theorem \ref{theorem : main}]
If $G$ is a finitely generated group with property A then we apply lemma \ref{lemma : main lemma} and lemma \ref{lemma : positive operator}
and the claim follows.

Conversely, given $\varepsilon>0$, the corresponding representation $\pi$ and a vector $v$ define $K(g,h)=\langle \pi_gv,\pi_hv\rangle$.
Then $K$ is positive definite and it is easy to check that it satisfies the conditions required by Theorem \ref{theorem : property A via p.d. kernels}.
\end{proof}

\subsection*{\textbf{A path of representations}}
Let $G$ be a finitely generated group. $G$ coarsely embeds into the Hilbert space $H$ 
if there exists a map $f:G\to H$, two non-decreasing functions $\rho_-,\rho_+:[0,\infty)\to [0,\infty)$
such that 
$$\rho_-(d(g,h))\le \Vert f(g)-f(h)\Vert_H\le \rho_+(d(g,h)),$$
and $\lim_{t\to \infty}\rho_-(t)=\infty$. Such an $f$ is called a coarse embedding. 

It is shown in \cite{guentner-kaminker} that if there exists $\theta>0$ such that $\rho_-(t)\ge Ct^{1/2+\theta}+D$
for $t\ge E$, for some constants $C,D,E>0$, then the positive definite kernel 
$$K_\alpha(g,h)=e^{-\alpha\Vert f(g)-f(h)\Vert_H^2},$$
induces a bounded positive operator on $\ell_2(G)$ for every $\alpha>0$. The proof relies on the Schur test. 
The existence of $\theta$ as above is strictly stronger than property A. 
Similarly as above we can use these
kernels to construct uniformly bounded representations.

Let $f:[0,\infty)\to [0,\infty)$ be a smooth function such that 
\begin{enumerate}
\item $\lim_{t\to 0}f(\alpha)=0$,
\item $\lim_{t\to\infty}f(\alpha)$ exists.
\end{enumerate}
Applying the previous construction to the operators
$$T_{\alpha}=K_\alpha+f(\alpha)I,$$
we obtain a family of representations $\set{\pi_\alpha}_{\alpha=0}^{\infty}$, that interpolates between the coefficients 
of the trivial representation at $\alpha=0$ 
and the left regular representation  at $\alpha=\infty$.

\section{Concluding remarks: Norms and strong property (T)}
It is natural to ask  how the norm $\Vert\pi\Vert$ of the representations in Theorem \ref{theorem : main}
behaves when $\varepsilon\to 0$. The norm of the uniformly bounded representation $\pi$ induced by renorming of a Hilbert space $H_0$ 
via a positive self-adjoint operator $T$ is the number
$$\Vert\pi \Vert=\inf \left\{ c\in [1,\infty)\ \Big\vert \begin{array}{ll} c^2T-\pi_{g^{-1}}T\pi_g \text{ is a positive operator }\\ \text{on }H_0 \text{ for every } g\in G\end{array}\right\rbrace.$$
Estimating the above norm does not seem to be an easy task.
Since the bottom of the spectrum $\lambda$ of $T$ converges toward zero as $\varepsilon\to 0$,
the right hand side of the  estimate $\Vert \pi\Vert\le \frac{\Vert T\Vert_{B(\ell_2(G))}}{\lambda}$ 
tends to infinity and it is natural to expect that the norms of $\pi$ will blow up to infinity as our 
coefficients of $\pi$ approach the
trivial representation. For some 
groups this cannot be improved. 

Consider the following strong version of property (T): \emph{$H^1(G,\pi)=0$ for any uniformly bounded representation
$\pi$ of $G$ on a Hilbert space}. 
Equivalently, every affine action with linear part $\pi$ given by a uniformly bounded
representation on a Hilbert space, has a fixed point.
This property is possessed by higher rank lattices [Shalom, unpublished], universal lattices \cite{mimura} and Gromov
monsters  \cite{naor-silberman}. As a consequence we have
\begin{proposition}
Let $G$ have the above strong property (T) for uniformly bounded representations. 
Then for any family of 
uniformly bounded representations $\pi$ satisfying Theorem \ref{theorem : main},
$\Vert\pi\Vert\to \infty$ as $\varepsilon\to 0$.
\end{proposition}

\begin{proof}
Assume the contrary. Then for every $\varepsilon>0$ there exists a uniformly bounded representation $\pi=\pi_{\varepsilon}$ and vectors $v_{\varepsilon}$, satisfying
the conditions of Theorem \ref{theorem : main}, with the additional property that $\sup\Vert\pi_\varepsilon\Vert\le C$
for some constant $C>0$.

Choosing a summable sequence of $\varepsilon$  we construct a Hilbert space 
$H=\bigoplus_{\varepsilon} H_\varepsilon$ and a representation 
$\rho=\bigoplus \pi_{\varepsilon}$. By the assumption on the uniform bound on norms
of $\pi_\varepsilon$ the representation $\rho$ is uniformly bounded on $H$.
Now construct a cocycle $b_g=\oplus (\pi_{\varepsilon})_gv_{\varepsilon}-v_{\varepsilon}$. 
Following the proof of \cite{bekka-cherix-valette}
we conclude that $b$ is a proper cocycle, in particular $b$ is not a coboundary.
\end{proof}

\end{document}